\documentclass[11pt,notitlepage,onecolumn]{article}

\usepackage[letterpaper]{geometry}
\usepackage[utf8]{inputenc}
\usepackage[OT4]{fontenc}
\usepackage{amsmath}
\usepackage{amsthm}
\usepackage{amsfonts}
\usepackage{amssymb}
\usepackage{verbatim}
\usepackage{bm}
\usepackage{bbm}

\usepackage[retainorgcmds]{IEEEtrantools}
\usepackage{marvosym}
\usepackage{stmaryrd}
\usepackage{tikz}
\usepackage{hyperref}
\usepackage{mdframed}
\usepackage{xcolor}
\usepackage{hhline}
\usepackage{thmtools}
\usepackage{thm-restate}
\usepackage{mathrsfs}
\usepackage{enumerate}

\theoremstyle{plain}
\newtheorem{theorem}{Theorem}[section]
\newtheorem{lemma}[theorem]{Lemma}

\theoremstyle{definition}
\newtheorem{definition}[theorem]{Definition}

\DeclareMathOperator*{\PP}{P}
\DeclareMathOperator*{\EE}{E}
\DeclareMathOperator*{\Var}{Var}
\DeclareMathOperator*{\Cov}{Cov}
\DeclareMathOperator*{\supp}{supp}

\DeclareMathOperator{\Inf}{Inf}

\DeclareMathOperator{\sgn}{sgn}

\newcommand{\cP}{\mathcal{P}}

\newcommand{\bbN}{\mathbb{N}}
\newcommand{\bbR}{\mathbb{R}}

\newcommand{\ignore}[1]{}

   {\medskip\begin{mdframed}\setlength{\parindent}{0cm}}%
   {\end{mdframed}\medskip}

\newcommand{\third}{{\frac{1}{3}}}
\newcommand{\half}{{\frac{1}{2}}}

\definecolor{DSgray}{cmyk}{0,0,0,0.7}
\definecolor{DSred}{cmyk}{0,0.7,0,0.7}

\newcommand{\eps}{{\varepsilon}}

\title{Gaussian Bounds for Noise Correlation of Resilient Functions }
\date{}

\author{Elchanan Mossel\thanks{MIT, Cambridge MA, USA.
E-mail: {\tt{elmos@mit.edu}}.
Partially supported by NSF Grant CCF 1665252 and DOD ONR grant N00014-17-1-2598}}

\begin{document}

\maketitle

\begin{abstract}
Gaussian bounds on noise correlation of functions play an important role in hardness of approximation, in quantitative social choice theory and in testing. 
The author (2008) obtained sharp gaussian bounds for the expected correlation of $\ell$ low influence functions $f^{(1)},\ldots, f^{(\ell)} : \Omega^n \to [0,1]$, where the inputs to the functions are correlated via the $n$-fold tensor of distribution $\mathcal{P}$ on $\Omega^{\ell}$. 

 It is natural to ask if the condition of low influences can be relaxed to the condition that the function has vanishing Fourier coefficients. Here we answer this question affirmatively. For the case of two functions $f$ and $g$, we further show that if $f,g$ 
 have a noisy inner product that exceeds the gaussian bound, then the Fourier supports of their large coefficients intersect. 

\end{abstract} 

\newpage

\section{Introduction}
Gaussian bounds on noise correlation of functions play an important role in hardness of approximation and quantitative social choice theory. A key example is the ``Majority is Stablest Theorem" proven~\cite{MoOdOl:05,MoOdOl:10}. 
Much more general bounds were obtained in~\cite{Mossel:08,Mossel:10}. The gaussian bounds of~~\cite{Mossel:08,Mossel:10}   are crucial ingredients in many hardness of approximation proofs, including~\cite{Raghavendra:08,GuMaRa:08,Chan:16,FeGuRaWu:12,Austrin:10,BansalKhot:10,GHMRC:11,AustrinHastad:11,AustrinHastad:12,GuRaSaWu:12}. 
They also play an important role in quantitative social choice, e.g.~\cite{IsakssonMossel:12,Mossel:12,FKKN:11}, in property testing, see e.g.~\cite{Blais:09,ChSeTa:14,ChSeTa:15}, and in mathematical analysis, 
see e.g.~\cite{NoGiGe:10,IsakssonMossel:12,Rossignol:16}. 

The gaussian bounds of~\cite{Mossel:08,Mossel:10} require that the functions have low influences. It is natural to ask 
if the low influence condition can be replaced by the weaker condition that all of the Fourier coefficients of $f$ are small.
Here we provide an affirmative answer to this question. 

In the reminder of the introduction we provide an informal statement of the main result of~\cite{Mossel:08,Mossel:10} followed by an informal statement of the main results of the current work. We then provide a simple example, discuss the motivation for this work, proof ideas and related work. 

We first recall the definition of influences. To simplify notation we will often omit the sigma algebra and probability measure defined over a probability space $\Omega$. For simplicity, in this paper 
we will only consider either random variables taking finite number of values or Gaussian random variables. 
\begin{definition}
Consider a probability space $\Omega$. 
For a function $f : \Omega^n \to \bbR$, we define the $i$'th influence of $f$ as 
\[
I_i(f) = \EE \big[\Var[f | x_1,\ldots,x_{i-1},x_{i+1},\ldots,x_n] \big],
\]
where the expected value is with respect to the product measure on $\Omega^n$. 
In the Boolean case with the uniform measure $f : \{-1,1\}^n \to \bbR$, the influence is equivalently defined as 
\[
I_i(f) = \EE \big[\Var[f | x_1,\ldots,x_{i-1},x_{i+1},\ldots,x_n] \big] 
= \sum_{S : i \in S} \hat{f}^2(S). 
\]
\end{definition} 

Our main interest is in strengthening the result of~\cite{Mossel:10} which provide Gaussian bounds for noise correlation of low-influence. We defer the definition of the more general setup to a later section and for now state the main result of~\cite{Mossel:10} informally. 

\begin{theorem}[\cite{Mossel:10} Informal Statement]
\label{thm:invariance_informal}
For every $\eps > 0$, there exists $\tau > 0$ for which the following holds. 
Consider a probability distribution $\mathcal{P}$ over $\Omega^{\ell}$ with $\rho(\mathcal{P}) < 1$. 
Let $\mathcal{G}$ denote a normal distribution over $R^{\ell |\Omega|}$ with the same covariance structure as $\mathcal{P}$. 
Let $f^{(1)},\ldots,f^{(\ell)} : \Omega^n \to [0,1]$ with al influences bounded by $\tau$:
\begin{equation} \label{eq:inf_bd0}
\max_{1 \leq j \leq \ell, 1 \leq i \leq n} I_i(j) \leq \tau.  
\end{equation} 
Then there exist functions $\phi^{(j)} : \bbR^{|\Omega|} \to [0,1]$ with $\EE[\phi^{(j)}] = \EE[f^{(j)}]$ such that 
\[
\EE \left[ \prod_{j=1}^\ell f^{(j)}(\underline{X}^{(j)}) \right] \leq 
  \EE[ \left[ \prod_{j=1}^\ell \phi^{(j)}(\underline{G}^{(j)}) \right] + \epsilon . 
\]
\end{theorem} 
In the equation above: 
\begin{itemize}
\item 
$\EE[f^{(j)}]$ is taken with respect to the $j$'th marginal measure of $\mathcal{P}$ to the $n$-th power. 
\item 
$\EE[\phi^{(j)}]$ it taken with respect to the $j$'th marginal measure of $\mathcal{G}$ to the $n$-th power.
\item $\EE \left[ \prod_{j=1}^\ell f^{(j)}(\underline{X}^{(j)}) \right]$ is taken with respect to the following measure, where 
for each $i$, the vector $(\underline{X_i}^{(j)})_{j=1}^{\ell}$ is distributed according to $\mathcal{P}$ independently. 
\item $\EE \left[ \prod_{j=1}^\ell \phi^{(j)}(\underline{G}^{(j)}) \right]$ is taken with respect to the measure where 
for each $i$, the vector $(\underline{G_i}^{(j)})_{j=1}^{\ell}$ is distributed according to $\mathcal{G}$ independently. 
\item The condition $\rho(\mathcal{P}) < 1$ is natural spectral-gap condition on the distribution $\mathcal{P}$.
\end{itemize}  

As mentioned earlier, Theorem~\ref{thm:invariance_informal} has many applications in many areas including in particular, hardness of approximation and social choice theory. A weaker condition than small Fourier coefficients is to require that $f$ is resilient. For a set $S \subset [n]$ and a vector $(X_i : i \in [n])$, we write $X_S = (X_i : i \in S)$. 

\begin{definition}
We say that a function $f : \Omega^n \to \bbR$ is $(r,\alpha)$-{\em resilient} if 
\begin{align} \label{eq:cond_res_intro}
\Big| \EE \left[f  | X_S = z \right] - \EE[f] \Big| \leq \alpha. 
\end{align}
for all $j$, all sets $S$ with $|S| \leq r$ and all $z$.  
then $f$ $(r,\alpha)$-resilient.)
\end{definition}
We note that in the case of Boolean functions under the uniform measure, if 
\[
\max (|\hat{f}(S)| : 0 < |S| \leq r) \leq 2^{-r} \alpha.
\]
then $f$ $(r,\alpha)$-resilient. 
 

\subsection{Main Results}
In our main result we provide an affirmative answer to the question above by proving:

\begin{theorem} \label{thm:main_informal}
(Informal statement): The conclusion of Theorem~\ref{thm:invariance_informal} holds if instead of assuming that the functions have low influences (\ref{eq:inf_bd0})  
we assume that all functions are $(r(\eps),\alpha(\eps))$-resilient for some $r(\eps) \in \bbN$ and $\alpha(\eps) > 0$. 
\end{theorem} 
See Theorem~\ref{thm:multi} for the formal statement with quantitative bounds.  

In the case of two function, it is enough that one of the functions is resilient to obtain the results of Theorem~\ref{thm:main_informal}. 
In fact, as we will see in Section~\ref{sec:intersect}, even when both functions are non resilient the statement holds, as long as their non-resilient variables ``do not intersect".  In particular we prove:

\begin{theorem} \label{thm:three_informal}
For every $\eps > 0, 0 \leq \rho < 1$, there exist $m,\beta> 0$ for which the following holds. 
\begin{itemize}
\item 
Let $x,y$ be $\rho$-correlated vectors. This means that 
 $( (x_i,y_i) : 1 \leq i \leq n)$ are i.i.d. mean $0$ ($\EE[x_i] = \EE[y_i] = 0$) and $\rho$-correlated ($\EE[x_i y_i] = \rho$). 
 \item 
Let $f,g : \{-1,1\}^n \to [0,1]$ be arbitrary functions and 
Let $f'(x) = 1(\sum x_i \geq a)$ and $g'(x) = 1(\sum x_i \geq b)$ with $\EE[f'] = \EE[f](1+o(1)), \EE[g'] = \EE[g](1+o(1))$, be symmetric threshold functions with expectations close to those of $f$ and $g$. 
\item  
Then if the correlation between $f$ and $g$ is $\eps$-larger than those of the corresponding threshold functions:  
\[
\EE[f(x) g(y)] \geq \EE[f'(x) g'(y)] + \eps,  
\]
\item 
Then there exists common structure in their Fourier spectrums: there exist $S, T \subset [n]$ such that 
\[
|S| \leq m, \quad,|T| \leq m, \quad S \cap T \neq \emptyset, \quad \min( |\hat{f}(S), \hat{g}(T) ) > \beta. 
\]
\end{itemize}
\end{theorem} 
See Theorem~\ref{thm:three} for a general statement with quantitative bounds. 

\subsection{An Example} 
To illustrate the results, we consider the following example.  
Let $( (X_i, Y_i, Z_i) )_{i=1}^n$ denote i.i.d. two dimensional vectors such that $X_i$ is distributed uniformly on $F_3$, the field of three elements and $Y_i = X_i$ with probability $1/2$ and $Y_i = X_i + 1 \bmod 3$ with probability $1/2$. Let $Z_i = Y_i$ with probability $1/2$ and $Y_i + 1 \bmod 3$ with probability $1/2$. 

Let $f, g, h: F_3^n \to \{0,1\}$ be subsets of $F_3^n$ and let $\mu_f, \mu_g, \mu_h$ denote the probabilities of these sets.
\begin{itemize}
\item By~\cite{Mossel:10} there are functions $\Delta_{\pm}(\mu_f, \mu_g, \mu_h)$ which are defined as supremum/infimum of Gaussian processes such that 
if $f,g$ and $h$ all have low influences then 
\begin{equation} \label{eq:example2}
\Delta_{-}(\mu_f, \mu_g, \mu_h) -  \eps \leq \EE[f(X)g(Y)h(Z)]  \leq \Delta_{+}(\mu_f, \mu_g, \mu_h) +  \eps
\end{equation}
where $\eps \to 0$ as the influences diminish. Moreover, both the upper bound and the lower bound are tight. 
\item By~\cite{HaHoMo:16}, 
\begin{equation} \label{eq:example2_2}
\EE[f(X)g(Y)h(Z)] \geq c(\mu_f,\mu_g,\mu_h) - \eps
\end{equation}
whenever $f,g$ and $h$ are resilient for some positive, dimension independent but for a non-optimal $c(\mu_f,\mu_g,\mu_h)$.
Moreover when $f=g=h$, (\ref{eq:example2_2}) holds without any conditions on $f$. 
\item The results of the current paper imply that the tight bound~(\ref{eq:example2}) holds whenever the three functions $f,g$ and $h$ are resilient where $\eps \to 0$ as $f,g,h$ becomes more resilient. 
\end{itemize}

\subsection{Motivation}
From a pure mathematical perspective, we argue that the Fourier transform of a function is a more fundemental object than its influence and thus our results are more natural. More concretely, this paper follows a line of 
work~\cite{Mossel:10,AustrinMossel:13,HaHoMo:16} 
 that tries to bridge between additive combinatorics and the theories of noise stability. One key difference between the two areas is that in additive combinatorics, obstructions to ``typical behavior" are stated in terms of Fourier or generalized Fourier norms, while in the theory of noise stability obstructions are in terms of high influences. 
In~\cite{AustrinMossel:13}, the authors succeeded to relax the influence condition from~\cite{Mossel:10} to Fourier norms for some pairwise independent distributions, while ~\cite{HaHoMo:16} obtained such a relaxation for essentially all distributions considered in the theory of noise stability (where a certain parameter $\rho$ is strictly less than $1$) but 
with lower bounds that are very far from tight. 

Additionally, in many applications, conditions on the Fourier spectrum are more natural than conditions on the influences. 
For example, from the noisy voting perspective, the statement that a function has a high influence variable means that there exists a voter $i$ than can have noticeable effect on the outcome {\em if voter $i$ has access to all other votes casted}. 
The statement that a function is not resilient implies that there is a bounded set of voters who have noticeable effect on the outcome {\em on average}, i.e., with no access to other votes casted. An example of a resilient function with a high influence variable is the function 
\[
f(x) = x_1 \sgn(\sum_{i=2}^n x_i).
\]
Here coordinate $1$ has influence $1$ but the function is resilient. 
In terms of voting, voter $1$ has a lot of power if she has access to all other votes casted (or the majority of the votes) but without access to this information she is powerless. Moreover, for $f$, every set of $k$ voters can change the expected value of $f$ (by conditioning on their vote) by $O(k n^{-1/2})$. 

Similar statements can be made in other voting contexts. For example, as we will show in Section~\ref{sec:arrow}, combining our results and Kalai's analysis~\cite{Kalai:02} implies that in the context of his paper 
(probability of Arrow's paradox among $3$ alternatives, using a balanced functions), if the probability of a paradox is $\eps$ smaller than the asymptotic value achieved by Majority, then at leas two of the functions are correlated with functions of  
$r(\eps)$ coordinates. We leave other social choice implications of our results to future work. 

A more challenging and exciting direction involves the potential application of our result to hardness of approximation. 
One difference between Unique Games Hardness proofs and NP-hardness proofs is that the former often require that passing a ``test" implies that a function has a large Fourier coefficient while for the later, it suffices to show that passing a ``test" implies that a function has a high influence variable. This is perhaps best exemplified by the difference between 
3CSPs~\cite{Hastad:97} and 2CSPs~\cite{Khot:02,KKMO:04}. 
Given our result, it is natural to explore if some Unique Games Hardness results can be transformed into NP-hardness results (we note that this difference in the ``inner verifier" is not the only difference and an equally or perhaps an even more challenging difference has to do with the ``outer verifier"). 

\subsection{Proof Ideas}
The basic idea of the proof is to look at a smoothed version of $f$ and construct a decision tree for it, where almost all leaves have low influences. If $f$ is resilient enough, it is not too hard to show that all leaves will have expectation that is very close to the expected value of $f$ which 
allows to apply the results of~\cite{Mossel:10}. In fact we will consider a number of different decision trees: 
\begin{itemize}
\item The proof of Theorem~\ref{thm:main_informal}  utilizes what we call a ``correlated decision tree", where at each node we ``expand" one variable in all functions simultaneously and where at most of the leaves, all functions have low influence. 
\item The proof of Theorem~\ref{thm:two} uses a decision tree for $g$ only so that for most leaves, $g$ is of low influence, while $f$ has essentially the same expectation as at the root.
\item  The proof of Theorem~\ref{thm:three_informal}  uses an additional decision tree for a general function $f$ such that in most of the leaves of $f$, $f$ is resilient.
\end{itemize}  
The idea of determining the resilience threshold as a function of an influence threshold 
is influenced by a recent work~\cite{HaHoMo:16}, where the same idea was used in the context of set-hitting. 

\subsection{Context and Related Work}

The results of~\cite{Mossel:08,Mossel:10} generalize the Majority is Stablest Theorem~\cite{MoOdOl:05,MoOdOl:10} (MIST). 
MIST was conjectured in the context of hardness of approximation 
by~\cite{KKMO:04,KKMO:07} and (in a special case) by Kalai in his work on Arrow's Theorem~\cite{Kalai:02}. 
The authors of~\cite{Kalai:02,KKMO:07,MoOdOl:10} asked if the 
condition of low influence can be replaced by the weaker condition that all of the Fourier coefficients of $f$ are small 
(see also~\cite{Filmus_etal:14})

After a draft of our paper was posted we learned from Ryan O'Donnell that the single function case of our results, i.e, the MIST case, was proved in an unpublished work by O'Donnell, Servedio, Tan and Wan~\cite{OSTW:10}. The results of ~\cite{OSTW:10} were not widely circulated and the question if Majority Is Stablest under small coefficients appears as an open problem in~\cite{Filmus_etal:14}. 
The argument of~\cite{OSTW:10} is based on a different decision tree than ours. The argument is presented in a paper by Jones~\cite{Jones:16} which was posted before our (independent) work was posted.  The argument presented in~\cite{Jones:16} uses the noise stability as an energy function. It is not clear if one can use this  energy function to obtain  stability bounds involving more than one function. 

\subsection{Paper structure}
In order to simplify the presentation of the ideas of the paper, we begin with a proof in the Boolean two-function case. Preliminaries are presented in 
Section~\ref{sec:prelim} while the proof of this special case, Theorem~\ref{thm:two}, is presented in Section~\ref{sec:proof}. The proof of Theorem~\ref{thm:main_informal} is provided in Section~\ref{sec:general}. 
In Section~\ref{sec:intersect}, we prove Theorem~\ref{thm:three_informal}. Finally in Section~\ref{sec:arrow} we demonstrate a direct application of our results which establishes a strengthening of the Arrow-Kalai Theorem. 
 
\subsection{Acknowledgment} 
Thanks to Vishesh Jain and Govind Ramnarayan for comments on a draft of this paper. 
Thanks to Ryan O'Donnell for communicating the existence of~\cite{OSTW:10,Jones:16} after the first draft of this paper was posted.

\section{Preliminaries} \label{sec:prelim}

\subsection{Two Function Version of Majority is Stablest} 
For simplicity we first discuss the original setting of Majority is Stablest (MIST from now on).
 In this case we consider $\{-1,1\}^n$ with the uniform probability measure. 
\begin{definition}
For two functions $f,g : \{-1,1\}^n \to \bbR$ the ($\rho$-) {\em noisy inner product} of $f$ and $g$ denoted by 
$\langle f, g \rangle_{\rho}$ 
is defined by $\EE[f(x) g(y)]$, where 
$( (x_i,y_i) : 1 \leq i \leq n)$ are i.i.d. mean $0$ ($\EE[x_i] = \EE[y_i] = 0$) and $\rho$-correlated ($\EE[x_i y_i] = \rho$). 
The {\em noise stability} of $f$ is its noisy inner product with itself: $\langle f, f \rangle_{\rho}$.
\end{definition} 

We can also write the noisy inner product in terms of the noise operator $T_{\rho}$, 
\[
\langle f, g \rangle_{\rho} = \EE[f(x) g(y)] = \EE[ f T_{\rho} g], \quad T_{\rho} f(x) = \EE[f(y) | x] = \sum_S \rho^s \hat{f}(S) x_S, \quad x_S = \prod_{i \in S} x_i.
\]
Analogous quantities are defined in Gaussian space. 
\begin{definition} 
In Gaussian space, the ($\rho$-) noisy inner product of $\phi : \bbR \to \bbR$ and $\psi : \bbR \to \bbR$ denote by $\langle \phi, \psi \rangle_{\rho}$  is 
\[
\EE[\phi(N) \psi(M)],
\]
where $M,N$ are standard (mean $0$ variance $1$) Gaussian random variables with covariance $\rho$ 
(so $\EE[N M] = \rho$). 
The {\em noise stability} of $\phi$ is its inner product with itself: $\langle \phi, \phi \rangle_{\rho}$.
\end{definition} 
We will generally use $f,g$ etc. to denote functions over the Boolean cube and $\phi,\psi$ etc. for functions over 
the one dimensional Gaussian measure. 
In particular, for $\mu \in [0,1]$ we write $\chi_{\mu}$ for the indicator of the interval $(-\infty,\Phi^{-1}(\mu))$ whose Gaussian measure is
$\mu$. We write $\mu_f$ (respectively $\mu_{\phi}$) for the expected value of $f$ under the uniform measure (respectively $\phi$ under the Gaussian measure). 

In our main result we replace the low-influence condition with the condition that the functions are {\em resilient}. . 
We first specialize to the Boolean case and recall the two functions version of the Majority Is Stablest Theorem (MIST):
\begin{theorem}[\cite{MoOdOl:05,MoOdOl:10}]
For every $\eps > 0, 0 \leq \rho < 1$, there exists a $\tau> 0$ for which the following holds. 
Let $f,g : \{-1,1\}^n \to [0,1]$ satisfying $\max(I_i(f),I_i(g)) < \tau$ for all $i$. 
Then 
\[
\langle f, g \rangle_{\rho} \leq \langle \chi_{\mu_f}, \chi_{\mu_g} \rangle_{\rho}+ \eps. 
\]
\end{theorem} 
This theorem is called Majority Is Stablest since
$
\langle \chi_{\mu_f}, \chi_{\mu_g} \rangle_{\rho} = \lim_{n \to \infty} \langle f_n , g_n \rangle_{\rho}, 
$
where $f_n(x) = \chi_{\mu_f}(n^{-1/2} \sum_{i=1}^n x_i)$ and $g_n(x) = \chi_{\mu_g}(n^{-1/2} \sum_{i=1}^n x_i)$. 

We need the following version from \cite{Mossel:10} which only requires one of the functions to have low influences.   
\begin{theorem}[\cite{Mossel:10} Prop 1.15] \label{thm:MIST}
For every $\eps > 0$ and $0 \leq \rho < 1$, there exists a $\tau(\rho,\eps) > 0$ for which the following holds. 
Let $f,g : \{-1,1\}^n \to [0,1]$ be such that $\min(I_i(f),I_i(g)) < \tau$ for all $i$. Then 
\begin{equation} \label{eq:mist}
\langle  f, g \rangle_{\rho} \leq \langle \chi_{\mu_f}, \chi_{\mu_g} \rangle_{\rho} + \eps, 
\end{equation} 
where one can take 
\begin{equation} \label{eq:tau_bound}
\tau = \eps^{O\left( \frac{\log(1/\eps) \log(1/(1-\rho))}{(1-\rho) \eps} \right)}. 
\end{equation} 
In particular the statement above holds when $\max_i I_i(f) < \tau$ and $g$ is {\em any} Boolean function bounded between $0$ and $1$. 
\end{theorem}

\subsection{Stability of Half-Spaces}
We will use the following standard estimates, see e.g.~\cite[Appendix B]{MoOdOl:10}.
\begin{lemma} \label{lem:gaussian_continuous} 
Assume $\rho < 1$ and $\rho_1 < \rho_2 < 1$ then 
\[
| \langle \chi_{\mu_1}, \chi_{\mu_2} \rangle_{\rho_1} - \langle  \chi_{\mu_1}, \chi_{\mu_2} \rangle_{\rho_2} | 
\leq    \frac{10 (\rho_2 - \rho_1)}{1-\rho_2}
\]

\[
| \langle  \chi_{\mu'_1}, \chi_{\mu'_2} \rangle_{\rho} - \langle  \chi_{\mu_1}, \chi_{\mu_2} \rangle_{\rho} | 
\leq   +  2 |\mu_1-\mu_1'|  + 2 |\mu_2-\mu_2'| 
\]
\end{lemma} 

We will also use the following corollary of Borell's Gaussian noise stability result. 
\begin{lemma}[\cite{Borell:85}] \label{lem:Borell} 
Let $\phi,\psi : \bbR \to [0,1]$ and $\rho \geq 0$ then 
\[
\langle \phi, \psi \rangle_{\rho} \leq \langle \chi_{\mu_\phi}, \chi_{\mu_\psi} \rangle_{\rho}. 
\]
\end{lemma}

\subsection{Decision Trees}
We will use the following regularity lemma, see e.g.~\cite{MosselSchramm:08,DKN:10}. 
\begin{lemma} \label{lem:dec}
For a function $f$, let $I(f) = \sum I_i(f)$. 
Then for any $\tau > 0, \eps > 0$ and any function $f$ there exists a decision tree for $f$ of depth at most 
\[
d \leq 2+ \frac{I}{\tau \eps} 
\]
such that except for at most a fraction $\eps$ of the leaves (chosen uniformly at random) all the leaf functions have their maximum influence at most 
$\tau$.
\end{lemma}

\begin{proof}
The construction of the decision tree is standard. If a function $f_{x_I}$ at a certain node $x_I$ has all influences less than $\tau$ or if the node is at level $d$ do nothing. Otherwise, condition on the variable $j$ with the maximum influence in $f_{x_I}$ and create two children $f_{y_J}$ and $f_{z_K}$ where $J = K = I \cup \{ j \}$, $y_i  = z_i = x_i$ for all $i \in I$ 
and $y_j =0$ and $z_j = 1$. Since 
\[
I(f_{x_I}) = I_j(f_{x_I}) + \frac{1}{2}(I(f_{y_J}) + I(f_{z_K})),
\]
it easily follows that if $L$ is the set of leaves of the tree and if $D(\ell)$ denotes the depth of leaf $\ell$ then 
\[
I(f) \geq \tau \sum_{\ell \in L} 2^{-\ell} D(\ell).
\]
Therefore if $p$ is the fraction of paths that reach level $d$ then 
\[
I(f) \geq (d-1) \tau p \implies p \leq I / (d-1) \tau,
\]
and taking $d -1$ to be the smallest integer that is greater or equal to $\frac{I}{\tau \eps}$ we obtain the desired result. 
\end{proof} 



\subsection{Statement of Special Case} 
We will first prove the following theorem which may be viewed as a special case of both 
Theorem~\ref{thm:multi} (as we only consider two functions) and of Theorem~\ref{thm:three} (as the resilience conditions we impose are stronger). 

\begin{theorem} \label{thm:two}
For every $\eps > 0, 0 \leq \rho < 1$, there exist $r,\alpha> 0$ for which the following holds. 
Let $f : \{-1,1\}^n \to [0,1]$ that is $(r,\alpha)$-resilient and let $g :  \{-1,1\}^n \to [0,1]$ be an arbitrary function. 
Then 
\[
\langle f, g \rangle_{\rho} \leq \langle \chi_{\mu_f}, \chi_{\mu_g} \rangle_{\rho}+ \eps. 
\]
One can take
\begin{equation} \label{eq:r_alpha}
r = O \left(\frac{1}{\eps^2(1-\rho)\tau} \right), \alpha = O\left(\eps 2^{-r} \right), 
\end{equation} 
where $\tau$ is given by (\ref{eq:tau_bound}).
\end{theorem} 
Note in particular that for our current bounds for $\tau$ and for 
fixed $\rho$, $r$ is exponential in a polynomial in $1/\eps$ and 
$\alpha$ is doubly exponential in a polynomial in $1/\eps$. 

\section{Proof of the special case} \label{sec:proof} 

The proof will be carried out via a sequence of reductions which will give the function $f$ more and more structure. 
Fix $\eps > 0$ and $0 \leq \rho' < 1$. 
Recall that we want to show that if $f$ is $(d(\eps,\rho'),\alpha(\eps,\rho'))$ resilient  then for all $g$ bounded between 
$0$ and $1$, 
\begin{equation} \label{eq:goal0}
\langle  f, g \rangle_{\rho'} \leq \langle \chi_{\mu_f}, \chi_{\mu_g} \rangle_{\rho'}+ \eps. 
\end{equation}

\begin{lemma} \label{lem:eta}
In order to prove~(\ref{eq:goal0}) it suffices to prove that 
\begin{equation} \label{eq:goal1}
\langle  T_{\eta} f, g \rangle_{\rho} \leq \langle  \chi_{\mu_f}, \chi_{\mu_g} \rangle_{\rho} + \eps/2. 
\end{equation}
for 
\begin{equation} \label{eq:rho_eta}
\rho = (1-0.01 \eps) \rho' + 0.01 \eps, \quad \eta = \rho'/\rho = 1- \Theta( \eps (1-\rho')). 
\end{equation} 
\end{lemma} 

\begin{proof} 
Write $\rho' = \rho \eta$, where $1-\rho \geq (1-\rho')/2$ and $\eta < 1$. Note that $f$ and $T_{\eta} f$ have the same 
expected value. If we could establish~(\ref{eq:goal1}) 
and  
\begin{equation} \label{eq:diff_chis}
|\langle  \chi_{\mu_f}, \chi_{\mu_g} \rangle_{\rho} - \langle  \chi_{\mu_f}, \chi_{\mu_g} \rangle_{\rho'}| < \eps/2, 
\end{equation} 
then~(\ref{eq:goal0}) would follow. Note that (\ref{eq:diff_chis}) follows from lemma~\ref{lem:gaussian_continuous} when 
\[
\frac{10 (\rho - \rho')}{1-\rho} < \eps/2.
\]
We may thus choose $\rho$ and $\eta$ as in~(\ref{eq:rho_eta}).  
\end{proof} 

\begin{lemma} \label{lem:dec_rep} 
Let $\tau$ be chosen so that (\ref{eq:mist}) holds with error $0.01 \eps$ for $\rho$.  
Then it suffices to prove (\ref{eq:goal1}) for a function $h = T_{\eta} f$ that has a decision tree of 
depth $d$ and such that for  at most $0.01 \eps$ fraction of the inputs a random path of the decision tree terminates at a node with some influence greater than $\tau$. Moreover
\[
d = O \left( \frac{1}{\eps^2 (1-\rho) \tau} \right)
\]
\end{lemma}

\begin{proof} 
We note that the function $h = T_{\eta} f$ satisfies: 
\[
I(h) := \sum I_i(h) = \sum_{S} |S| \hat{f}^2(S) \eta^{2 |S|} \leq \max_s s \eta^{2 s} = O\left( \frac{1}{\eps (1-\rho)} \right).  
\]
Apply lemma~\ref{lem:dec} to obtain a decision tree for $h$ where for at most $0.01 \eps$ fraction of the inputs, a random path of the decision tree terminates at a node with some influence greater than $\tau$. Note that the depth of the tree satisfies
\[
d \leq C (1+  \frac{I}{\tau \eps}) \leq C(1 + \frac{1}{\eps^2 (1-\rho) \tau})
\] 
as needed.
\end{proof} 

We now conclude the proof of Theorem~\ref{thm:two}. 
\begin{proof} 
Let $h$ be a function such as in lemma~\ref{lem:dec_rep}. Assume furthermore that 
$f$ is $(d,0.01 \eps 2^{-d})$ resilient. Note that this implies that $h  = T_{\eta} f$ is also $(d,0.01 \eps 2^{-d})$ resilient. 
Let $x,y$ be two $\rho$-correlated inputs. 
Then
\[
\EE[h(x) g(y)] = \sum_{x_I,y_I} \PP[x_I] \PP[y_I | x_I] \EE[h(x) g(y) | x_I, y_I],
\] 
where $x_I$ denotes a random leaf of the decision tree and $y_I$ is chosen after $x_I$ to be a $\rho$-correlated 
version of $x_I$. Let $A$ denote the set of $x_I$ for which $f_{x_I}$ has all influences less than $\tau$. Then:
\begin{eqnarray*}
\EE[h(x) g(y)] &=& \sum_{x_I,y_I} \PP[x_I] \PP[y_I | x_I] \EE[h(x) g(y) | x_I, y_I] 
\\ &\leq& 0.01 \eps + 
\sum_{x_I \in A} \PP[x_I] \sum_{y_I} \PP[y_I | x_I] \EE[h(x) g(y) | x_I, y_I]. 
\end{eqnarray*}
Write $\mu' = \mu_f + 0.01\eps$ and $\mu(y_I) = \EE[g(y) | y_I]$. 
Note that since $h$ is $(d,0.01 \eps 2^{-d})$-resilient 
it follows that for all leaves 
$x_I$ it holds that $\EE[h | x_I] \leq \mu'$. Thus for  $x_I \in A$ we can apply (\ref{eq:mist})  to obtain that 
\[
\EE[h(x) g(y) | x_I, y_I] \leq \langle  \chi_{\mu'},  \chi_{\mu(y_I)} \rangle_{\rho} + 0.01 \eps.
\]
Plugging this back in we obtain the bound
\begin{eqnarray*}
\EE[h(x) g(y)] &\leq&  
0.02 \eps + \sum_{x_I \in A} \PP[x_I] \sum_{y_I} \PP[y_I | x_I] \langle  \chi_{\mu'},  \chi_{\mu(y_I)} \rangle_{\rho} \\ 
&\leq& 0.02 \eps + \sum_{x_I,y_I} \PP[x_I] \PP[y_I | x_I] \langle  \chi_{\mu'}, \chi_{\mu(y_I)} \rangle_{\rho} \\
&=& 0.02 \eps +  \langle  \chi_{\mu'},  (\sum_{x_I,y_I} \PP[x_I] \PP[y_I | x_I] \chi_{\mu(y_I)}) \rangle_{\rho}
\end{eqnarray*}
Note that $\psi = \sum_{x_I,y_I} \PP[x_I] \PP[y_I | x_I] \chi_{\mu(y_I)}$ is a $[0,1]$-valued function with
$\EE[\psi] = \EE[g]$. Thus by lemma~\ref{lem:Borell} it follows that 
\[
0.02 \eps +  \langle  \chi_{\mu'}, (\sum_{x_I,y_I} \PP[x_I] \PP[y_I | x_I] \chi_{\mu(y_I)}) \rangle_{\rho} \leq 
0.02 \eps + \langle  \chi_{\mu'},  \chi_{\mu_g} \rangle_{\rho} \leq 0.04 \eps +  \langle  \chi_{\mu_f},  \chi_{\mu_g} \rangle_{\rho},
\]
where the last inequality follows from lemma~\ref{lem:gaussian_continuous}. 
\end{proof} 

\section{The general case} \label{sec:general}

In this section we discuss how to obtain the general result of the paper in the context discussed in~\cite{Mossel:10,HaHoMo:16}. 
We begin by recalling the setup and some results as formulated in~\cite{HaHoMo:16}.  

\subsection{Setup}
We begin with formally defining this general setting. 
Let $\Omega$ be a finite set  
and assume we are given a probability distribution $\mathcal{P}$
over $\Omega^\ell$
for some $\ell \ge 2$ -- we will call it an
\emph{$\ell$-step probability distribution over $\Omega$}.

Furthermore, assume we are given $n \in \bbN$.
We consider $\ell$ vectors $\underline{X}^{(1)}, \allowbreak \ldots, 
\allowbreak \underline{X}^{(\ell)}$,
$\underline{X}^{(j)} = (X_1^{(j)},\allowbreak \ldots, X_n^{(j)})$ such that 
for every $i \in [n]$, the $\ell$-tuple $(X_i^{(1)}, \ldots, \allowbreak  X_i^{(\ell)})$ 
is sampled according to $\mathcal{P}$, independently of the other
coordinates $i' \neq i$ 
(see Figure~\ref{fig:naming} for an overview of the notation). 
We write $\pi_1,\ldots,\pi_{\ell}$ for the $\ell$ marginals of $\pi$ and 
$\pi_{\ast}$ for the minimum non-zero marginal probability. We write $\EE_i$ for the expected value with respect to 
$\pi_i$ and power of $\pi_i$. 

We will write $F$ for a vector of $\ell$ functions $f^{(1)},\ldots,f^{(\ell)}$, where $f^{(j)} : \Omega^n \to \bbR$. 
We will write 
\[
\langle F \rangle_{\mathcal{P}} := \EE \left[ \prod_{j=1}^\ell f^{(j)}(\underline{X}^{(j)}) \right]
\]
and
\[
\EE[F] = \left( \EE_1[f^{(1)}],\ldots,\EE_{\ell}[f^{(\ell)}] \right)
\]

The Gaussian version of the distribution $\mathcal{P}$, denoted by $\mathcal{G}$, is a Gaussian distribution on $\bbR^{|\Omega| k}$ that has the same first two moments as the joint distribution of 
\[
\Big( 1 \big( (X_i^{(1)}, \ldots, \allowbreak  X_i^{(\ell)}) = (\omega_1,\ldots,\omega_{\ell}) \big) : (\omega_1,\ldots,\omega_k) \in \Omega^k \Big).
\]
Such a distribution is constructed more explicitly in~\cite{Mossel:10,HaHoMo:16} but the details of the construction are not needed here. Given $n \in \bbN$, we consider again $\ell$ vectors $\underline{N}^{(1)}, \allowbreak \ldots, 
\allowbreak \underline{N}^{(\ell)}$,
$\underline{N}^{(j)} = (N_1^{(j)},\allowbreak \ldots, N_n^{(j)})$ such that 
for every $i \in [n]$, the $\ell$-tuple $(N_i^{(1)}, \ldots, \allowbreak  N_i^{(\ell)})$ 
is sampled according to $\mathcal{G}$, independently of the other
coordinates $i' \neq i$ (following the same notation as Figure~\ref{fig:naming}). 
Again, we will write $\Phi$ for a collection of $\ell$ functions: $\phi^{(1)},\ldots,\phi^{(\ell)} : \bbR^{|\Omega| k} \to [0,1]$ and write 
\[
\langle \Phi \rangle_{\mathcal{G}} := \EE \left[ \prod_{j=1}^\ell \phi^{(j)}(\underline{G}^{(j)}) \right]
\]
Given $\mu = (\mu_1,\ldots,\mu_{\ell}) \in [0,1]^{\ell}$, we define 
\[
\Gamma_{\mathcal{G}}(\mu) = \sup_{n, \phi}  \Big( \langle \Phi \rangle_{\mathcal{G}}  \; : \;  \forall j\,:\, \phi_j : \bbR \to [0,1], \EE[\Phi] = \mu \Big)
\]
It is immediate to see that 
\begin{equation} \label{eq:diffGamma}
|\Gamma_{\mathcal{G}}(\mu) - \Gamma_{\mathcal{G}}(\nu)| \leq \sum_{i=1}^n |\mu_i - \nu_i|.
\end{equation} 

\begin{figure}
\begin{tikzpicture}[yscale=-0.7]
\node(X__) at (0,0) {$\underline{\overline{X}}$};
\node(X1_) at (1,0) {$\overline{X}_1$};
\node(X2_) at (2,0) {$\overline{X}_2$};
\node(X3_) at (3,0) {$\dots$};
\node(X4_) at (4,0) {$\overline{X}_i$};
\node(X5_) at (5,0) {$\dots$};
\node(X6_) at (6,0) {$\overline{X}_n$};

\node(X_1) at (0,1) {$\underline{X}^{(1)}$};
\node(X11) at (1,1) {$X_1^{(1)}$};
\node(X21) at (2,1) {$X_2^{(1)}$};
\node(X31) at (3,1) {$\cdots$};
\node(X41) at (4,1) {$X_i^{(1)}$};
\node(X51) at (5,1) {$\cdots$};
\node(X51) at (6,1) {$X_n^{(1)}$};

\node(X_2) at (0,2) {$\underline{X}^{(2)}$};
\node(X12) at (1,2) {$X_1^{(2)}$};
\node(X22) at (2,2) {$X_2^{(2)}$};
\node(X32) at (3,2) {$\cdots$};
\node(X42) at (4,2) {$X_i^{(2)}$};
\node(X52) at (5,2) {$\cdots$};
\node(X52) at (6,2) {$X_n^{(2)}$};

\node(X_3) at (0,3) {$\vdots$};
\node(X13) at (1,3) {$\vdots$};
\node(X23) at (2,3) {$\vdots$};
\node(X43) at (4,3) {$\vdots$};
\node(X63) at (6,3) {$\vdots$};

\node(X_4) at (0,4) {$\underline{X}^{(j)}$};
\node(X14) at (1,4) {$X_1^{(j)}$};
\node(X24) at (2,4) {$X_2^{(j)}$};
\node(X34) at (3,4) {$\cdots$};
\node(X44) at (4,4) {$X_i^{(j)}$};
\node(X54) at (5,4) {$\cdots$};
\node(X64) at (6,4) {$X_n^{(j)}$};

\node(X_5) at (0,5) {$\vdots$};
\node(X15) at (1,5) {$\vdots$};
\node(X25) at (2,5) {$\vdots$};
\node(X45) at (4,5) {$\vdots$};
\node(X65) at (6,5) {$\vdots$};

\node(X_6) at (0,6) {$\underline{X}^{(\ell)}$};
\node(X16) at (1,6) {$X_1^{(\ell)}$};
\node(X26) at (2,6) {$X_2^{(\ell)}$};
\node(X36) at (3,6) {$\cdots$};
\node(X46) at (4,6) {$X_i^{(\ell)}$};
\node(X56) at (5,6) {$\cdots$};
\node(X66) at (6,6) {$X_n^{(\ell)}$};

\draw[-] (0.5,-.5) to (0.5,6.5);
\draw[-] (-.5,0.5) to (6.5,0.5);

\node(iid) at (4,-3) {\begin{minipage}{4.2cm}\raggedright
Tuples $\overline{X}_{i}$ are 
i.i.d.~according to $\mathcal{P}$.
The marginals of $\mathcal{P}$ are~$\pi_j$.\end{minipage}};

\draw[->] (iid) to (X1_.north);
\draw[->] (iid) to (X2_.north);
\draw[->] (iid) to (X4_.north);
\draw[->] (iid) to (X6_.north);

\node[rotate=0](rowdistr) at (-2.5, 4) 
{\begin{minipage}{3cm}\raggedright Vectors $\underline{X}^{(j)}$
are distributed according to $\underline{\pi_j} := \pi_j^n$.
\end{minipage}};
\draw[->] (rowdistr) to (X_1.west);
\draw[->] (rowdistr) to (X_2.west);
\draw[->] (rowdistr) to (X_4.west);
\draw[->] (rowdistr) to (X_6.west);

\node(commonDist) at (-1, -3) 
{\begin{minipage}{3.2cm}\raggedright Distributed according to 
$\underline{\mathcal{P}}:=\mathcal{P}^n$.\end{minipage}};
\draw[->] (commonDist) to (X__.north west);

\node(propertiesofP) at (0,9) {\begin{minipage}{5cm}\begin{align*}
\alpha(\mathcal{P}) &:= \min_{x \in \Omega} \mathcal{P}(x,x,\ldots,x)\\
\rho(\mathcal{P}) &: \text{See Definition~\ref{def:correlation}}
\end{align*}\end{minipage}};

\node(domainsofxij) at (6,10) {\begin{minipage}{4cm}\begin{align*}
X_{i}^{(j)} &\in \Omega\\
\underline{X}^{(j)} &\in \underline{\Omega} := \Omega^{n}\\
\overline{X}_{i} & \in \overline{\Omega} := \Omega^{\ell}\\
\underline{\overline{X}} & \in \underline{\overline{\Omega}} 
:= \Omega^{n\cdot \ell}\\
S &\subseteq \underline{\Omega}
\end{align*}\end{minipage}};

\end{tikzpicture}
\caption{Naming of the random variables in the general case.
The columns $\overline{X}_i$ are distributed $i.i.d$ according to
$\mathcal{P}$. 
Each $X_{i}^{(j)}$ is distributed according to $\pi_j$.
The overall distribution of $\overline{\underline{X}}$ is
$\underline{\mathcal{P}}$.}
\label{fig:naming}
\end{figure}
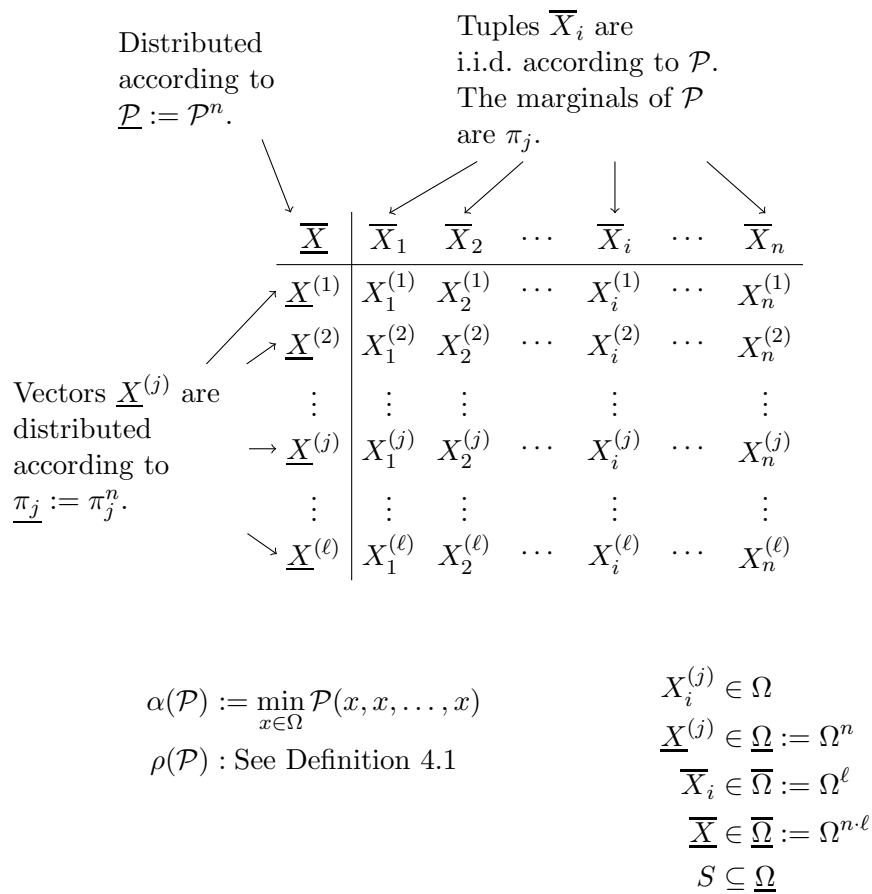

We will also require the following definition, see also~\cite{Mossel:10}:
\begin{definition}\label{def:correlation}
Let $\mathcal{P}$ be a single-coordinate distribution
and let $S, T \subseteq [\ell]$. We define the \emph{correlation}:
\begin{align*}
\rho(\mathcal{P}, S, T) &:=
\sup \Bigl\{ \Cov[f(X^{(S)}), g(X^{(T)})] \Bigm| 
f: \Omega^{(S)} \to \mathbb{R}, 
g: \Omega^{(T)} \to \mathbb{R}, \\
&\qquad\qquad
\Var[f(X^{(S)})] = \Var[g(X^{(T)})] = 1 
\Bigr\} \; .
\end{align*}
The correlation of $\mathcal{P}$ is
$\rho(\mathcal{P}) := \max_{j \in [\ell]} 
\rho\left(\mathcal{P}, \{j\}, [\ell]\setminus\{j\}\right)$.
\end{definition}
As in previous results~\cite{Mossel:10,HaHoMo:16}, we will require that $\rho(\mathcal{P})$ is strictly less than $1$.
The opposite condition $\rho(\mathcal{P}) = 1$ is equivalent
to the following:
There exists $j \in [\ell]$, $S \subseteq \Omega$, 
$T \subseteq \Omega^{\ell-1}$ such that
$0 < |S| < |\Omega|$ and:
\begin{align*}
X_i^{(j)} \in S \iff
\left(X_i^{(1)}, \ldots, X_i^{(j-1)}, X_i^{(j+1)}, \ldots, X_i^{(\ell)}\right)
\in T \; .
\end{align*}

\subsection{Prior Results} 
We will need two results from~\cite{HaHoMo:16} which are small variations of the results of~\cite{Mossel:10}.
In the discussion below we will not distinguish between a function $f$ on a finite product space and its representation as a multi-linear polynomial. 
The first results shows that condition that $\rho < 1$ implies that we can smooth functions. 

\begin{theorem}[\cite{Mossel:10}, \cite{HaHoMo:16} A.69]
\label{thm:smoothing}
Let $\overline{\underline{X}}$ be a random vector distributed according
to $(\overline{\underline{\Omega}}, \underline{\mathcal{P}})$
with $\rho(\overline{\Omega}, \mathcal{P}) \le \rho \le 1$.

Let $\epsilon \in (0, 1/2]$ and 
$\gamma \in \left[0, \frac{(1-\rho)\epsilon}{\ell \ln \ell/\epsilon}\right]$.

Then, for all $F = (f^{(1)}, \ldots, f^{(\ell)})$ taking values in $[0,1]^{\ell}$ it holds that 
\begin{align*}
  \left| \langle F \rangle_{\mathcal{P}} - \langle T_{1-\gamma} F \rangle_{\mathcal{P}} \right| \le \epsilon \; .
\end{align*}
\end{theorem}
The formal statement in~\cite{HaHoMo:16} requires that all marginals of $\mathcal{P}$ are equal but the proof does not require this fact. 

The second result is the main result of~\cite{Mossel:10}. The statement below follows from the proof of Theorem 4.1 in~\cite{HaHoMo:16}. 

\begin{theorem}
\label{thm:invariance}
Let $\overline{\underline{X}}$ be a random vector distributed
according to $(\overline{\underline{\Omega}}, \underline{\mathcal{P}})$ 
such that $\cP$ has $\rho(\mathcal{P}) \le \rho < 1$.

Then, for all $\epsilon > 0$, there exists 
$\tau := \tau(\epsilon, \rho, \pi_{\ast}, \ell) > 0$
such that if functions 
$f^{(1)}, \ldots, f^{(\ell)}: \underline{\Omega} \to [0, 1]$ satisfy
\begin{align}
\max_{i \in [n], j \in [\ell]} \Inf_i(f^{(j)}(\underline{X}^{(j)})) \le \tau\;,
\end{align}
then, there exist functions $\Phi \to [0,1]^{\ell}$ with $\EE[\Phi] = \EE[F]$ such that 
\begin{align*}
\left| \langle F \rangle_{\mathcal{P}} - \langle \Phi \rangle_{\mathcal{G}} \right| \le \epsilon \; .
\end{align*}

Furthermore, there exists an absolute constant $C \ge 0$ such that
for $\epsilon \in (0, 1/2]$ one can take
\begin{align}\label{eq:38a}
\tau := \left(\frac{(1-\rho^2)\epsilon}{\ell^{5/2}}
\right)^{C \frac{\ell \ln(\ell/\epsilon)\ln(1/\pi_{\ast})}{(1-\rho)\epsilon}} 
\; .
\end{align}
\end{theorem}

\subsection{Our Results} 
We can now state our results.

\begin{theorem}
\label{thm:multi}
Let $\overline{\underline{X}}$ be a random vector distributed
according to $(\overline{\underline{\Omega}}, \underline{\mathcal{P}})$ 
with $\rho(\mathcal{P}) \le \rho < 1$.

Then, for all $\epsilon > 0$, there exists an 
 $r = r(\epsilon, \rho, \pi_{\ast}, \ell)$ such that the following holds. 
Let 
$f^{(1)}, \ldots, f^{(\ell)}: \underline{\Omega} \to [0, 1]$ be $(r,\eps/4 \ell)$ resilient, i.e., 
\begin{align} \label{eq:cond_res}
\Big| \EE \left[f^{(j)} | (\underline{X_i}^{(j)} : i \in S) = z \right] - \EE[f] \Big| \leq \frac{\eps}{4 \ell},
\end{align}
for all $j$, all sets $S$ with $|S| \leq r$ and all $z$.  
Then  
\begin{align*}
 \langle F \rangle_{\mathcal{P}} \leq \Gamma_{\mathcal{G}}(\mu) +  \eps. 
\end{align*}

Furthermore, there exists an absolute constant $C \ge 0$ such that
for $\epsilon \in (0, 1/2]$ one can take
\begin{align}\label{eq:r_main}
r := \frac{4 \ell^2 \ln \ell/\epsilon}{\tau (1-\rho)\epsilon^2},
\end{align}
where $\tau$ is given in (\ref{eq:38a}). 
\end{theorem}

\subsection{The correlated decision tree} 
\begin{definition}
Given a vector of functions  $F= (f^{(1)},\ldots,f^{(\ell)})$, we define the {\em correlated decision tree} of $F$ as follows.   
Vertices of the trees are vectors of functions denotes $F_{\overline{X}_S = z}$ where $S \subset [n]$. 
$F_{\overline{X}_S = z}$ is the function $F$ where we restrict the values of $( \overline{X}_i : i \in S)$ to take the value $z$. 

If such a vertex is not a leaf, then it has $|\Omega|^k$ children of the form $F_{\overline{X}_T = z}$, where $T$ is a super-set of $S$ that contains one more one element and $z'$ ranges over all possible value with $z'_S = z$. 
The root of the tree is $F = F_{\emptyset}$. We write $I(F) = \sum_{i,j} I_i(f^{(j)})$. 
\end{definition} 

We can easily generalize lemma~\ref{lem:dec} to the following
\begin{lemma} \label{lem:dec2}
For any $\tau > 0, \eps > 0$ and $F$ as above there exists a decision tree of depth $d$  
\[
d \leq 4+ \frac{I(F)}{\tau \eps} 
\]
such that except for at most a fraction $\eps$ of the leaves (according to the $\mathcal{P}$ distribution) all leaves
$F_{\overline{X}_T = z}$ have that all of the coordinates have all influences bounded by $\tau$. 
\end{lemma}

We can now prove Theorem~\ref{thm:multi}.
\begin{proof}
Fix $\eps > 0$. By Theorem~\ref{thm:smoothing} it suffices to prove the statement of the Theorem with error $\eps/2$ for the functions
$T_{1-\gamma} f^{(i)}$ where 
\[
\gamma = \frac{(1-\rho)\epsilon}{4 \ell \ln \ell/\epsilon}.
\]
We will slightly abuse notation are write $f^{(i)}$ for $T_{\gamma} f^{(i)}$. We also write $F = (f^{(1)},\ldots,f^{(\ell)})$. 
Consider a decision tree for $F$ of depth $d$ such that except for $0.01 \eps$ of the leaves, all leaf function have all influences bounded by $\tau$, where  
$\tau$ is given in~(\ref{eq:38a}) for error $0.01 \eps$ (this can be achieved by choosing a different value of $C$ in (\ref{eq:38a})). 
We note that 
\[
I(F) \leq \frac{4 \ell^2 \ln \ell/\epsilon}{(1-\rho)\epsilon}. 
\]
Using the decision tree from lemma~\ref{lem:dec2} with 
\[
d \leq 4+ \frac{I(F)}{\tau \eps} \leq 4 + \frac{4 \ell^2 \ln \ell/\epsilon}{\tau (1-\rho)\epsilon^2}, 
\]
we next evaluate 
\begin{equation} \label{eq:multi_de}
\EE \left[ \prod_{j=1}^\ell f^{(j)}(\underline{X}^{(j)}) \right] = 
\EE \Big[ \EE \left[ \prod_{j=1}^\ell f^{(j)}(\underline{X}^{(j)})  \Big| \overline{X}_S = z \right] \Big].
\end{equation} 
For all leaves of the tree the conditional expectation is bounded by $1$ and except with probability $0.01 \eps$, the vector function 
$F_{\overline{X}_T = z}$ has all influences bounded by $\tau$. For such functions by Theorem~\ref{thm:invariance} it holds that 
\begin{eqnarray*}
 \EE \left[ \prod_{j=1}^\ell f^{(j)}(\underline{X}^{(j)}) \Big| \overline{X}_S = z \right] &\leq&  
 \Gamma_{\mathcal{G}}\left( \EE[f^{(j)}(\underline{X}^{(j)}) | \overline{X}_S = z] : 1 \leq j \leq \ell \right) \\ &\leq&  
 \Gamma_{\mathcal{G}}\left( \EE[f^{(j)}]: 1 \leq j \leq \ell \right) + \eps/4,  
\end{eqnarray*}
where the last inequality uses the fact that the functions $f^{(j)}$ are resilient, i.e., (\ref{eq:cond_res}) and (\ref{eq:diffGamma}). 
It thus follows that the LHS of (\ref{eq:multi_de}) is bounded by 
\[
\Gamma_{\mathcal{G}}\left( \EE[f^{(j)}]: 1 \leq j \leq \ell \right) + \eps/4 + 0.01 \eps < \Gamma_{\mathcal{G}}\left( \EE[f^{(j)}]: 1 \leq j \leq \ell \right) + \eps/2,
\]
as needed. 
\end{proof}

\subsection{Two Functions}
While the results of this section are more general than the special case considered earlier, they are weaker in requiring that all functions are resilient.
For completeness we state the following theorem whose proof is essentially identical to the proof of Theorem~\ref{thm:two}.

\begin{theorem} \label{thm:two_general}
For every $\eps > 0, 0 \leq \rho < 1$, there exist $r,\alpha> 0$ for which the following holds. 
Let $\overline{\underline{X}} = (\underline{X}^{(1)},\underline{X}^{(2)})$ be a random vector distributed
according to $(\overline{\underline{\Omega}}, \underline{\mathcal{P}})$ 
with $\rho(\mathcal{P}) \le \rho < 1$.
Let $f : \{-1,1\}^n \to [0,1]$ that is $(r,\alpha)$-resilient and let $g :  \{-1,1\}^n \to [0,1]$ be an arbitrary function. 
Then 
\[
\langle f, g \rangle_{\mathcal{P}} \leq \langle \chi_{\mu_f}, \chi_{\mu_g} \rangle_{\rho}+ \eps. 
\]
One can take
\begin{equation} \label{eq:r_alpha}
r = O \left(\frac{1}{\eps^2(1-\rho)\tau} \right), \alpha = O\left(\eps 2^{-r} \right), 
\end{equation} 
where on can take 
\begin{equation} \label{eq:tau_general}
\tau = \eps^{O\left( \frac{\log(1/\mu) \log(1/\eps) \log(1/(1-\rho))}{(1-\rho) \eps} \right)},
\end{equation} 
where $\mu$ is the minimal weight of an atom of $\Omega$. 
\end{theorem} 


\section{Interaction Between Fourier Spectrums} \label{sec:intersect} 

\subsection{The Fourier Decision Tree} 

The proof below will require the {\em Steele-Efron-Stein} decomposition of 
$f : \Omega^n \to \bbR$: 
\[
f = \sum_{S \subset [n]} f_S,
\]
whose properties can be found for example in~\cite{Mossel:10}. We recall that $f_{x_S}$ denote the function $f$ where the variables in 
the set $S$ are restricted to the values $x_S$, while $f_S$ is a function of all the variables in $S$. 

\begin{definition}
Given parameters $(r,\alpha)$ and a function $f : \Omega^n \to \bbR$, the $(r,\alpha)$-{\em Fourier support of} $f$ is all variables $i$ such that $i \in S$,
where $|S| \leq r$ and such that $\Var[f_S] \geq \alpha^2$: 
\[
\supp_{(r,\alpha)}(f) := \{ i : \exists S \subset [n], i \in S, 0 < |S| \leq r, \Var[f_S] \geq \alpha^2 \}.
\]
We say that two functions $f$ and $g$ are $(m,\alpha)$-{\em cross-resilient} if 
\[
\supp_{(r,\alpha)}(f) \cap \supp_{(r,\alpha)}(g) = \emptyset.
\]
\end{definition} 

Note that if $f$ is $(r,\alpha)$-resilient and $0 < |S| \leq r$ then 
\[
Var[f_S] \leq \frac{1}{4} (\max f_S - \min f_S)^2 \leq \max \big( (\max f_S - \EE[f])^2, (\EE[f] - \min f_S)^2 \big) \leq \alpha^2.
\]
Therefore $(f,g)$ are $(r,\alpha)$-cross-resilient for any $g$. 
Next we will strengthen Theorem~\ref{thm:two} to prove the following. 

\begin{theorem} \label{thm:three}
For every $\eps > 0, 0 \leq \rho < 1$, there exist $m,\beta> 0$ for which the following holds. 
Let $f,g : \{-1,1\}^n \to [0,1]$ such that $f$ and $g$ are $(m,\beta)$-cross-resilient 
Then 
\[
\langle f, g \rangle_{\rho} \leq \langle \chi_{\mu_f}, \chi_{\mu_g} \rangle_{\rho}+ \eps. 
\]
One can take
\[
m = O(r/\eps \alpha^2), \quad \beta = \Omega (\mu^m 2^{-m} \eps),
\] 
where $r$ is chosen as in~(\ref{eq:r_main}). 
\end{theorem}

\begin{lemma} \label{lem:dec_fourier}
Let $\Omega$ be a finite probability space with minimum atom probability $\pi_{\ast}$. 
Let $f : \Omega^n \to [0,1], r \in \bbN$ and $\alpha > 0, \eps > 0$. Then there exists a Fourier decision tree for $f$ with 
the following properties:
\begin{itemize}
\item The depth of the tree $d$ is a most $r(1 + \frac{1}{\alpha^2 \eps})$.
\item Given a random path from the root, the probability that it ends at an $(m,\alpha)$-resilient function is at least $1-\eps$. 
\item If $f_{x_S}$ is a node of the decision tree and $i \in S$ then $i \in \supp_{(d, \pi_{\ast}^d 2^{-d} \alpha)}(f)$. 
\end{itemize} 
\end{lemma} 

\begin{proof}
We expand a node $f_{x_S}$ in the decision tree, if in its Efron-Stein-Steele decomposition, there is a set $T \subset [n] \setminus S$, such that 
$\Var[f_{x_S, T}] \geq \alpha^2$.  
Note that if we expand the node $f_{x_S}$ and its possible extensions are $(f_{z_{S \cup T}})$ then the conditional variance formula implies that 
\begin{eqnarray*}
\Var[f_{x_S}] & = & \Var[\EE[f_{z_{S \cup T}}]] + \EE[\Var[f_{z_{S \cup T}}]] = \sum_{T' \subseteq T} \Var[f_{x_S,T'}]  + \EE[\Var[f_{z_{S \cup T}}]] \\ 
&\geq& \Var[f_{x_S,T}] + \EE[\Var[f_{z_{S \cup T}}]] \geq \alpha^2 + \EE[\Var[f_{z_{S \cup T}}]].
\end{eqnarray*}
It easily follows that if $L$ is the set of leaves of the tree and if $D(\ell)$ denotes the depth of leaf $\ell$ then 
\[
\Var[f] \geq \alpha^2 \sum_{\ell \in L} 2^{-\ell} D(\ell/r).
\]
Therefore if $p$ is the fraction of paths that reach level $d$ then 
\[
1 \geq \Var[f] \geq (d/m-1) \tau \alpha^2 \implies p \leq \frac{1}{ (d/r-1) \alpha^2},
\]
and all other leaves are $(r,\alpha)$-resilient. 
Taking $d -1$ to be the smallest integer that is greater or equal to $\frac{m}{\alpha^2 \eps}$, we obtain the desired 
bound on the depth of the tree: $d \leq r(1 + \frac{1}{\alpha^2 \eps})$. 

Consider a new node added to the tree: $f_{z_{S \cup T}}$. This node was added since 
$\Var[f_{x_S,T}] \geq \alpha^2$. Note that if the original Efron-Stein-Steele decomposition of $f = \sum_U f_S$, then 
\[ 
f_{x_S,T}(x_{\setminus S}) = \sum_{U : T \subseteq U \subseteq T \cup S}  f_U(x_S,x_{\setminus S}),
\]
and therefore 
\[
\Var[f_{x_S,T}(x)] = \Var[\sum_{U : T \subseteq U \subseteq T \cup S}  f_U(x_S,x_{\setminus S})] \leq 
2^d \sum_{U : T \subseteq U \subseteq T \cup S} \Var[ f_U(x_S,x_{\setminus S})]. 
\] 
Therefore at least one of the terms $\Var[ f_U(x_S,x_{\setminus S})]$ is bounded below by 
$2^{-d} \alpha^2$. 
Using the conditional variance formula again, we get that 
\[
\Var[ f_U] \geq \pi_{\ast}^d \Var[ f_U(x_S,x_{\setminus S})], 
\]
and the proof of the lemma follows.

\end{proof}

We now prove Theorem~\ref{thm:three}. The proof is similar to the proof of Theorem~\ref{thm:two}, where we replace the application of Theorem~\ref{thm:MIST} with Theorem~\ref{thm:two} and the influence decision tree with the Fourier decision tree. 

\begin{proof} 
Choose $\alpha$ and $r$ such that the conclusion of Theorem~\ref{thm:two_general} holds for a function 
$f$ that is $(r,\alpha)$ resilient with error bounded by $0.01 \eps$. 
Apply lemma~\ref{lem:dec_fourier} to construct a Fourier decision tree for $f$ where except for $0.01 \eps$ of the leaves, all the leaves are $(r,\alpha)$-resilient. Note that the depth $d$ of the tree is bounded by 
$m = O(r/\eps \alpha^2)$. We now claim that the statement of the theorem holds if $f,g$ are $(m,\beta)$-cross-resilient, 
where $\beta = \min(\pi_{\ast}^m 2^{-m} \alpha,0.01 \times \pi_{\ast}^m 2^{-m} \eps)$. 

The key observation is that  for every path $x_I$ in the decision tree for $f$ and for every 
$i \in I$, it holds that $i \in \supp{(m, \beta)}(f)$ and therefore $i \notin \supp{(m, \beta)}(g)$. Note that 
\begin{equation} \label{eq:g_res}
\big| \EE[g | y_I] - \EE[g] \big| = \big| \sum_{\emptyset \subsetneq S \subset I}  g_S(y_I) \big| \leq 2^{d} \max_{\emptyset \subsetneq S \subset I} |g_S(y_I)|,
\end{equation}
and therefore if the difference is of magnitude at least $\eta$ then there exists an $S$ such that $|g_s(y_I)| \geq 2^{-d}$ which in turn implies that
\[
\Var[g_S] = \EE[g_S^2] \geq 2^{-2d} \eta^2 \pi_{\ast}^d = 10^{-4} 2^{-2d} \eps^2 \pi_{\ast}^d. 
\]
However, the cross resilience condition implies that 
$\Var[g_S] \leq \beta^2$ and therefore for $\beta^2 = O(2^{-d} \eps \pi_{\ast}^{d})$, it holds that 
\[
|\EE[g | y_I] - \EE[g]| \leq 0.01 \eps,
\]
for all leaves of the decision tree. 

We now expand:
\[
\EE[f(x) g(y)] = \sum_{x_I,y_I} \PP[x_I] \PP[y_I | x_I] \EE[f(x) g(y) | x_I, y_I],
\] 
where $x_I$ denotes a random leaf of the decision tree and $y_I$ is chosen after $x_I$ to be a $\rho$-correlated 
version of $x_I$. Let $A$ denote the set of $x_I$ for which $f_{x_I}$ is $(m,\alpha)$-resilient. Then:
\begin{eqnarray*}
\EE[f(x) g(y)] &=& \sum_{x_I,y_I} \PP[x_I] \PP[y_I | x_I] \EE[f(x) g(y) | x_I, y_I] 
\\ &\leq& 0.01 \eps + 
\sum_{x_I \in A} \PP[x_I] \sum_{y_I} \PP[y_I | x_I] \EE[f(x) g(y) | x_I, y_I]. 
\end{eqnarray*}
Write $\mu' = \mu_g + 0.01\eps$ and $\mu(x_I) = \EE[f(x) | x_I]$. 

By~(\ref{eq:g_res}), it follows that for  $x_I \in A$ we can apply Theorem~\ref{thm:two_general}  to obtain that 
\[
\EE[f(x) g(y) | x_I, y_I] \leq \langle  \chi_{\mu(x_I)}, \chi_{\mu'} \rangle_{\rho} + 0.01 \eps.
\]
Plugging this back in we obtain the bound
\begin{eqnarray*}
\EE[f(x) g(y)] &\leq&  
0.02 \eps + \sum_{x_I \in A} \PP[x_I] \sum_{y_I} \PP[y_I | x_I]  \langle  \chi_{\mu(x_I)},  \mu' \rangle_{\rho}  \\ 
&\leq& 0.02 \eps + \sum_{x_I,y_I} \PP[x_I] \PP[y_I | x_I]   \langle  \chi_{\mu(x_I)},  \mu' \rangle_{\rho} \\
&=& 0.02 \eps +  \langle  \chi_{\mu'},  (\sum_{x_I,y_I} \PP[x_I] \PP[y_I | x_I] \chi_{\mu(x_I)}) \rangle_{\rho}.
\end{eqnarray*}
Note that $\psi = \sum_{x_I,y_I} \PP[x_I] \PP[y_I | x_I] \chi_{\mu(x_I)}$ is a $[0,1]$-valued function with
$\EE[\psi] = \EE[f]$. Thus by lemma~\ref{lem:Borell} it follows that 
\[
0.02 \eps +  \langle  \chi_{\mu'}, (\sum_{x_I,y_I} \PP[x_I] \PP[y_I | x_I] \chi_{\mu(x_I)}) \rangle_{\rho} \leq 
0.02 \eps + \langle  \chi_{\mu'},  \chi_{\mu_f} \rangle_{\rho} \leq 0.04 \eps +  \langle  \chi_{\mu_g},  \chi_{\mu_f} \rangle_{\rho},
\]
where the last inequality follows from lemma~\ref{lem:gaussian_continuous}. 
\end{proof}


\section{A Strengthening Of Arrow-Kalai Majority Theorem} \label{sec:arrow}

A combination of Kalai's analysis of a probabilistic version of Arrow's Theorem and MIST implies the following: 

\begin{theorem}[ ~\cite{Kalai:02,MoOdOl:10}]\label{thm:kalai}
Consider $( (x_i,y_i,z_i) : 1 \leq i \leq n)$ i.i.d. drawn from the uniform distribution on 
$\{-1,1\}^3 \setminus \{ \pm (1,1,1) \}$. For every $\eps > 0$, there exists a $\tau > 0$ such that 
if $f,g,h : \{-1,1\}^n \to \{0,1\}$ satisfy $\EE[f] = \EE[g] = \EE[h] = 1/2$ and have all influences bounded above by $\tau$ then:
\begin{equation} \label{eq:kalai}
\PP[f(x) = g(y) = h(z)] \geq  3 \langle \chi_{\half}, 1-\chi_{\half} \rangle_{\third} -  \half + \eps. 
\end{equation}
\end{theorem} 
The interpretation in terms of voting is that $(x_i,y,_i,z_i)$ represents voter $i$ pairwise preferences between 
alternative A vs. B, B vs. C and C vs. A. Since the voter is assumed to be rational the preferences $(1,1,1)$ and 
$-(1,1,1)$ are not feasible. If voter $i$ votes at random, then its preferences are uniformly distributed among the remaining possible preference profiles. The function $f$ aggregates the individual A and B preferences and results in 
the societal preference between A and B (similarly for $g$ and $h$ ; for presentation purposes it is useful to encode these preferences using $0,1$ rather than $1,-1$). The expression $\PP[f(x) = g(y) = h(z)]$ is the probability that the outcome of the vote does not correspond to a ranking of $A,B$ and $C$. This is called a paradoxical outcome. 
The right hand side of equation~(\ref{eq:kalai}) is the asymptotic probability 
that $\PP[f(x) = g(y) = h(z)]$ where $f = g = h = \chi_{\half}(n^{-\half} \sum_{i=1}^n x_i)$ are all given by the same Majority function. 
 
Given our result we can prove the following strengthening of the result 
\begin{theorem}\label{thm:kalai2}
Consider $( (x_i,y_i,z_i) : 1 \leq i \leq n)$ i.i.d. drawn from the uniform distribution on 
$\{-1,1\}^3 \setminus \{ \pm (1,1,1) \}$. For every $\eps > 0$, there exists a $m, \beta > 0$ such that 
if $f,g,h : \{-1,1\}^n \to \{-1,1\}$ satisfy $\EE[f] = \EE[g] = \EE[h] = 0$ and each pair of functions among $f,g,h$ is 
$(m,\beta)$-cross-resilient then 
\[
\PP[f(x) = g(y) = h(z)] \geq  3 \langle \chi_{\half}, 1-\chi_{\half} \rangle_{\third} -  \half + \eps.
\]
\end{theorem} 

While Theorem~\ref{thm:kalai} requires that no voter can have a noticeable effect on any pairwise preference given its knowledge of all other votes casted, Theorem~\ref{thm:kalai2} requires much less. For the theorem not to hold, at least two of the functions, let us say $f,g$ have to be correlated with function $f', g'$ of a bounded ($m$) number of voters.
Moreover, there has to exist a voter $i$ that effects both the outcome of $f'$ and in $g'$.  

The proof of Theorem~\ref{thm:kalai2} is identical to the proof of Theorem~\ref{thm:kalai}. 
Here we need a lower bound version of Theorem~\ref{thm:three} saying that under the conditions of the theorem
\begin{equation} \label{eq:borell_lower}
\langle f, g \rangle_{\rho} \geq \langle \chi_{\mu_f}, 1 -\chi_{1-\mu_g} \rangle_{\rho} + \eps. 
\end{equation}
The lower bound proof is identical to the proof of the upper bound. 
Then we write: 
\begin{eqnarray*}
\PP[f(x) = g(y) = h(z)] &=&1 + \EE[f(x) g(y)] + \EE[g(y) h(z)] + \EE[h(z) f(x)] - \EE[f(x)] - \EE[g(y)] - \EE[(h(z)] \\ 
 &=&   \langle f, g \rangle_{-\third} + \langle g, h \rangle_{-\third} + \langle h, f \rangle_{-\third} - \half.
\end{eqnarray*} 
If $g'(x) = g(-x)$ then $\langle f, g \rangle_{-\third} = \langle f, g' \rangle_{\third}$ and the proof follows from~(\ref{eq:borell_lower}).

\bibliographystyle{abbrv}
\bibliography{all,my}

\end{document}